\documentclass[12pt]{amsart}
\usepackage{amssymb,latexsym,amsmath,amsthm,amscd}
\usepackage{setspace}
\usepackage[all]{xy} \xyoption{arc}
\usepackage[left=3cm,top=2cm,right=3cm,bottom = 2cm]{geometry}
\usepackage{graphicx}
\usepackage[usenames,dvipsnames]{color}
\usepackage{charter}

\theoremstyle{plain}
\newtheorem{thm}{Theorem}
\newtheorem{lem}[thm]{Lemma}

\theoremstyle{definition}

\theoremstyle{remark}
\newtheorem{rmk}[thm]{Remark}

\newcommand{\CC}{\mathbf{C}}
\newcommand{\ZZ}{\mathbf{Z}}
\newcommand{\QQ}{\mathbf{Q}}

\begin{document}
\title{Vertex Operator Algebras with two simple modules - the Mathur-Mukhi-Sen Theorem Revisited}
\author{Geoffrey Mason, Kiyokazu Nagatomo and Yuichi Sakai}
\address{Department of Mathematics\\
University of California, Santa Cruz}
\email{gem@ucsc.edu}
\address{Department of Pure and applied Mathematics, Graduate School of Information Science and Technology\\ Osaka University}
\email{nagatomo@math.sci.oswaka-u.ac.jp}
\address{Multiple Zeta Research Center, Kyushu University\\
Motooka 744, Nishi-ku, Fukuoka 819-0395, JAPAN}
\email{e-mail: dynamixaxs@gmail.com}

\thanks{The first author was supported by the Simons Foundation \#427007}
\thanks{The second author was partially supported  by JSPS KAKENHI Grant Number JP17K05171, 
 and by the Max-Planck Institute f\"{u}r Mathematik, Germany and the International Center for Theoretical Physics, Italy.}


\begin{abstract}  Let $V$ be a strongly regular vertex operator algebra and  let $\frak{ch}_V$ be the space spanned by the characters of
the irreducible $V$-modules.\ It is known that $\frak{ch}_V$ is the space of solutions of a so-called \emph{modular linear differential equation (MLDE)}.\ In this paper we obtain a near-classification of those $V$ for which the corresponding MLDE
is irreducible and monic of order $2$.\ As a consequence we derive the complete classification when $V$ has exactly two simple modules.\ It turns out that $V$ is either one of four affine Kac-Moody algebras of level $1$, or the Yang-Lee Virasoro model of central charge 
${-}22/5$.\ Our proof establishes new connections between the characters of $V$ and Gauss hypergeometric series, and puts the finishing touches to work of Mathur, Mukhi and Sen who first considered this problem forty years ago.  \\
\end{abstract}
\maketitle
 
 \section{Introduction}
 In a remarkable paper that was ahead of its time~\cite{MMS}, Mathur, Mukhi and Sen put forward the idea 
of classifying two-dimensional conformal field theories according to the differential equation satisfied 
by the characters of the simple modules (primary fields at vacuum). These differential equations, now call called MLDEs 
(\textit{modular linear differential equations}), are polynomials in the so-called Serre derivation with coefficients which are modular forms. (Further details are given below.)
 
 \medskip
 Mathur, Mukhi and Sen pushed through their ideas in the basic case when the MLDE has order two and is \textit{monic}
(leading coefficient 1)  corresponding to some theories containing just two two primary fields. They achieved a classification result in this case, however some ambiguities remained and their methods are mathematically incomplete.

\medskip
The purpose of the present paper is to revisit the classification of Mathur, Mukhi and Sen. Taking advantage of recent advances 
in the theory of MLDEs, in particular the connections with \textit{Gauss hypergeometric series}~\cite{FM1}, and also  the theory of rational vertex operator algebras~\cite{M2}, we obtain a complete result that is mathematically rigorous. In particular,
we give a new description of the characters of the modules of several familiar VOAs in terms of hypergeometric series. 
Our main results are stated below as Main Theorems 1 and 2.  In the rest of the Introduction we give a more detailed discussion 
of our results and methods of proof.
 
 \medskip
 The setting for our results is the theory of \textit{rational vertex operator algebras} (VOAs), and in particular VOAs~$V$
 that are \textit{strongly regular}. Informally, this means that $V$ is well-behaved. 
 An~overview of the theory can be found in~\cite{M2}.  However MLDEs are not treated in~\cite{M2}, and we discuss them here because they figure prominently in the present work.

\medskip
The weight 2 ``Eisenstein'' series is
\begin{eqnarray*}
E_2(\tau):=1-24\sum_{n=1}^{\infty}\sum_{d{\mid}n}dq^n.
\end{eqnarray*}
Here, and below, $\tau$ lies in the complex upper-half plane $H$ and $q:=e^{2\pi i \tau}$
The series $E_2(\tau)$ is  holomorphic
throughout $H$.\ Its main importance for us is its occurrence in the 
differential operators (sometimes called ``Serre'', or ``Ramanujan'', derivatives)
 \begin{eqnarray*}
D_k:= \frac{1}{2\pi i}\frac{d}{d\tau}-\frac{k}{12}E_2(\tau)= q\frac{d}{dq}-\frac{k}{12}E_2(\tau)\quad
(k{\in}\mathbf{Z}).
\end{eqnarray*}
The operator $D_k$ acts on the space $\mathcal{F}$ of holomorphic functions in $H$ 
and in this regard it has a basic invariance property. To describe this, let
$\Gamma:=SL_2(\mathbf{Z})$ be the homogeneous modular group.  For a given $k$, $\Gamma$ acts on the right 
of $\mathcal{F}$ by the $k^{\text{th}}$ \textit{stroke operator}
\begin{eqnarray*}
f|_k\gamma(\tau):= (c\tau+d)^{-k}f(\gamma\tau)\quad\text{for}\quad f\in\mathcal{F}\quad\text{and}\quad
\gamma:=\left(\begin{array}{cc}a & b \\ c & d\end{array}\right)  {\in}\Gamma.
\end{eqnarray*}
Then we have (\cite[Chapter X, \S\,5]{L})
\begin{eqnarray}\label{slashinv}
D_k(f)|_{k+2}\gamma(\tau)=D_k(f|_k\gamma)(\tau).
\end{eqnarray}

An MLDE is a differential equation of the form
\begin{eqnarray}\label{MLDEX}
\left(P_{0}D_k^n+P_{1}D_k^{n-1}+\cdots+P_{n-1}D_k+P_n\right)f=0
\end{eqnarray}
Here,  each $P_j$ is a holomorphic modular form on $\Gamma$ of weight $m-(k+2n-2j){\geq}0$ for some given integer $m$, 
and $D_k^n:=D_{k+2n-2}\circ\cdots\circ D_{k+2}\circ D_k$ (for additional details, see~\cite{MMS} and~\cite{M3}).
For example, the simplest  MLDE of order 2, 
having $n=2, k:=0,m:=4$, is
\begin{eqnarray}\label{ord2}
\left(D_0^2+k_1E_4(\tau)\right)f=0\quad(k_1{\in}\mathbf{C})
\end{eqnarray}
where $E_4(\tau)$ is the standard weight 4 Eisenstein series on $\Gamma$.

Eq.~\eqref{MLDEX} may be rewritten as a~(complicated) traditional linear differential equation involving the derivatives of $f$,
but this will not be useful for us. The formulation~\eqref{MLDEX} together with~\eqref{slashinv} makes it clear that the \
textit{solution space} is invariant under the stroke action $|_m$ of $\Gamma$, and this representation of the modular group 
is the \textit{monodromy} of the MLDE~\cite{FM2}.

Suppose that $V$ is a strongly regular VOA. Let $\frak{ch}_V$ denote the $\mathbf{C}$-linear space spanned by the
$q$-characters of the irreducible $V$-modules.  An important theorem of Zhu (\cite{Z}, \cite{DLM}) states 
that $\frak{ch}_V{\subseteq}\mathcal{F}$, moreover $\frak{ch}_V$ is a $\Gamma$-submodule with respect to 
the zeroth stroke operator $|_0$. Furthermore, we may use the \textit{modular Wronskian} (\cite{M3}, \cite{MMS}) together with Zhu's theorem to show that 
if $\dim\frak{ch}_V=n$ then $\frak{ch}_V$ is the solution space of some MLDE (\ref{MLDEX}) of order $n$.

In this way,  given a strongly regular $V$ with space of characters $\frak{ch}_V$, we obtain some important arithmetic/representation-theoretic data, namely the representation $\rho:\Gamma{\rightarrow}GL(\frak{ch}_V)$, and an MLDE with monodromy $\rho$.
This is related to, though rather different from, the usual $S$- and $T$- matrices of rational conformal field 
theories (RCFTs).

Mathur, Mukhi and Sen proposed to \textit{classify} RCFTs according to the MLDE associated to them. 
They considered in detail the case where the MLDE is~\eqref{ord2} and the monodromy $\rho$ is \textit{irreducible}.
 In particular $\dim\frak{ch}_V=2$.  This means that either
 
 \smallskip\noindent
 (i) $V$ has \textit{exactly two} irreducible modules and they have linearly independent characters, or else\\
 (ii) $V$ has more than two irreducible modules and their characters are \textit{not linearly independent}.

\smallskip
Inasmuch as an irreducible module and its dual have identical characters, the second possibility is commonplace.

\smallskip
In this paper we give a rigorous and full account of the classification of strongly regular VOAS in case (i).
We prove that are exactly five isomorphism classes of such VOAs, under the assumption that the monodromy $\rho$ is irreducible. Most of our analysis also applies to case (ii), although there are some cases (of central charges $c=-6, -8$ and $-10$) 
which remain open.

It is convenient to record here the assumptions and notation relating to $V$ that we will be operating under:
\begin{eqnarray*}
&&\ \ \ \ \ \ -\ \mbox{$V$ is a strongly regular, simple vertex operator algebra of central charge $c$.}\\
&&\ \ \ \  \ \ -\ \mbox{$\frak{ch}_V$ is the space of $q$-characters of the irreducible $V$-modules.}\\
&&(*)\  -\ \mbox{$V$ has an irreducible module $M$ of conformal weight $h$ and  the $q$-characters}\\
&&\ \ \ \ \ \ \  \mbox{$Z_V(\tau):=q^{-c/24}\sum_{n\geq 0}\dim V_nq^n$ and $Z_M(\tau):=q^{h-c/24}\sum_{n\geq 0}\dim M_{h+n}q^n$ span $\frak{ch}_V$}. \\
&&\ \ \ \ \ \ \  -\ \mbox{$\frak{ch}_V$ is the solution space of the order 2, monic MLDE~\eqref{ord2} and the associated}\\
&&\ \ \ \ \ \ \ \  \mbox{monodromy representation $\rho$ is \textit{irreducible}}.
\end{eqnarray*}

In order to describe our main results, we recall (cf~ \cite{M}) that the \textit{Gauss hypergeometric series} is the function
\begin{eqnarray}\label{hypergeo1}
{_2}F_1(a', b', c'; z):=1+\sum_{n\geq 0} \frac{(a')_n(b')_n}{(c')_n}\frac{z^n}{n!},
\end{eqnarray}
where $(a')_n$ is the \textit{Pochhammer symbol} (or rising factorial)
\begin{eqnarray*}
(a')_n:= a'(a'+1)(a'+2)\dotsb (a'+n-1).
\end{eqnarray*}
The series ${_2}F_1(a', b', c'; z)$ converges for all $a', b', c'{\in}\CC$ unless $c'$ is a nonpositive integer.\ It is a solution of the \textit{Gauss hypergeometric differential equation}
 \begin{eqnarray*}
\frac{d^2f}{dz^2}+\frac{(c'-(a'+b'+1)z)}{z(1-z)}\frac{df}{dz}-\frac{a'b'}{z(1-z)}f=0.
\end{eqnarray*}

The following two Theorems are the main results of the paper:
 
\medskip\noindent
 {\bf Main Theorem 1.}\label{thm1} 
 Suppose that $V$ is a vertex operator algebra satisfying the assumptions $(*)$.
 If $c{\geq}0$ then $V$ is isomorphic to one of  seven affine algebras of level~1:
 \begin{eqnarray*}
L_{A_1}(1, 0)), L_{A_2}(1, 0), L_{G_2}(1, 0), L_{F_4}(1, 0), L_{D_4}(1, 0), L_{E_6}(1, 0),  L_{E_7}(1, 0).
\end{eqnarray*}

If $c<0$ then either $V$ is isomorphic to the \textit{Yang-Lee model}, i.e., the discrete series Virasoro algebra $Vir_{c_{2, 5}}$ of central charge $-22/5$; or $V$ is one of a series of (unkown) VOAs of central charge $c=-6, -8$ or $-10$.

In all cases both known and unknown,  $Z_V(\tau)$ and $Z_M(\tau)$ are modular functions of weight~0 on a congruence subgroup
of $SL_2(\ZZ)$ and they may be described (up to an overall scalar) in terms of a~pair of rational numbers $(a, b)$ 
and the  Gauss hypergeometric series as follows:
\begin{equation*}
Z_V(\tau)= K^a\cdot   {_2}F_1(a, a+1/3, 2a+5/6; K),\quad
Z_M(\tau)= K^b\cdot {_2}F_1(b, b+1/3, 2b+5/6; K),
\end{equation*}
where $K$ is the level~1 hauptmodul on $\Gamma$ defined by
\begin{eqnarray*}
K:=\dfrac{1728}{j}:=\dfrac{E_4^3(\tau)-E_6^2(\tau)}{E_4^3(\tau)},
\end{eqnarray*}
according to the cases in Table~\ref{AandB}.

\begin{table}[htbp]
\begin{center}
\begin{tabular}{l|c|c|c}
Type&$a$&$b$&$c$\\
\hline
$A_1$&$-1/24$&$5/24$&$1$\\
$A_2$&$-1/12$&$1/4$ &$2$ \\
$G_2$&$-7/6$&$17/60$&$14/5$\\
$D_4$&$-1/6$&$1/3$&$4$\\
$F_4$&$-13/60$&$23/60$&$26/5$\\
$E_6$&$-1/4$&$5/12$&$6$\\
$E_7$&$-7/24$&$11/24$&$7$\\
$Vir_{c_{2, 5}}$&$11/60$&$-1/60$&$-22/5$\\
$??$&$1/4$&$-1/12$&$-6$\\
$??$&$1/3$&$-1/6$&$-8$\\
$??$&$5/12$&$-1/4$&$-10$\\
\end{tabular}
\end{center}
\caption{Values of $a,b,c$}
\label{AandB}
\end{table}

With a~slightly stronger hypothesis the unknown cases of Main Theorem~\ref{thm1} do \textit{not} exist:

\medskip\noindent
 {\bf Main Theorem 2.}\label{mainthm2} Suppose that $V$ is a~vertex operator algebra satisfying the assumptions $(*)$, 
 and suppose further that (up to isomorphism) $V$ and $M$ are the \textit{only} simple $V$-modules. Then $V$ is
 isomorphic to one of the following five VOAs:
\begin{eqnarray*}
L_{A_1}(1, 0)),\; \, L_{G_2}(1, 0),\;\, L_{F_4}(1, 0),\; \, L_{E_7}(1, 0),\;\, \mathsf{Vir}_{c_{2, 5}}\,.
\end{eqnarray*}

In our approach to the proofs of the Main Theorems, we first show that the~$q$-characters of all irreducible modules 
are \textit{modular functions on a~congruence subgroup}.\ In the present situation we are able to prove this famous modular-invariance
result based on recent advances in the theory of MLDEs~\cite{FM1}.\ We then closely consider the MLDE~\eqref{MLDEX}: we use 
a~ detailed knowledge of 2-dimensional congruence representations of $\Gamma$~\cite{M1} to show that there are only~9 possibilities for the monodromy~$\rho$.\ The description of the solutions of the MLDE in terms of Gauss hypergeometric series 
was given in~\cite{FM1}, and this result is fundamental to our approach.\ We use it to show that there are only \textit{finitely many} 
(14 in fact) possible values of the central charge $c$ (and the \textit{effective central charge} $\tilde{c}$) for a~VOA satisfying the assumptions of the Main Theorems.\footnote{This finiteness result is somewhat surprising because, for example, 
there is no analogous result in dimension~3: there are \textit{infinitely many} strongly regular VOAs with $\dim\frak{ch}_V=3$, 
and their $c$-values are unbounded.} These are listed in Table~\ref{Table:residual}.\ Our task is then to classify the VOAs according to this data.\ We use a number of classification results in the literature (summarized in Theorem~\ref{thmomni}) to show
that of the fourteen possible sets of data, some \textit{cannot} correspond to a~VOA, while others characterize the VOAs uniquely.

All of these results are obtained in Section~\ref{SPMT1} under the assumptions of Main Theorem~\ref{thm1},  
in particular $V$ may have more than two simple modules, although $\frak{ch}_V$ is always assumed to have dimension~2.\ However there are three values of $c$ which we cannot handle by these methods.\ To deal with these residual cases we must assume that $V$ has exactly two irreducible modules.\ The reason for this is that we can then use our modular-invariance result 
to \textit{explicitly identify} the $q$-characters as modular functions, and in particular we can write down the explicit $S$-matrix using known transformation laws for the modular functions in question.\ This is carried out in Section~\ref{SPMT2}.\ In each case we obtain 
the curious contradiction that the $S$-matrix is \textit{not symmetric}, thereby contradicting a basic fact of RCFT~\cite{Hu}, and then Main Theorem~\ref{mainthm2} is a~consequence.

\section{Proof of Main Theorem 1}\label{SPMT1}
In this Section we discuss the proof of  Main Theorem $1$.

\subsection{Modularity}\label{Modularity}
 In this Subsection we do \textit{not} need to assume that $\rho$ is irreducible. We will prove
 
\begin{thm}\label{thmmodular} 
Let $V$ be a~vertex operator algebra satisfying the conditions $(*)$ and let
$\rho{:}\Gamma{\rightarrow} GL(\frak{ch}_V)$ be the representation of $\Gamma$ furnished by the zeroth stroke action $|_0$.\ Then $\rho$ is \textit{modular}, i.e., $\ker\rho$ is a \textit{congruence subgroup} of $\Gamma$.\
In particular, both $Z_V(\tau)$ and $Z_M(\tau)$ are modular functions of weight~0 on a congruence subgroup of $SL_2(\ZZ)$.
\end{thm}
\begin{proof} 
Let $M(\rho)$ denote the space of holomorphic vector-valued modular forms corresponding to $\rho$.
This space is naturally $\ZZ$-graded by weight~$k$:
 \begin{eqnarray*}
M(\rho)=\bigoplus_{k\in\ZZ} M_k(\rho).
\end{eqnarray*}

We assert that there is $F(\tau){\in} M_k(\rho)$ for some integral weight $k$ such that $F(\tau)$ has \textit{bounded denominators}. Indeed, we may take
$F(\tau):=\Delta(\tau)^k W(\tau)$ for some $k$, where 
\begin{eqnarray*}
W(\tau):=\left(\begin{array}{c}Z_V(\tau) \\ Z_M(\tau)\end{array}\right)
\end{eqnarray*}
is the meromorphic vector-valued modular form defined by $V$.\ (There are no poles in $H$, but there may be poles at the cusps.)\
This assertion 
follows because $W(\tau)$ has integral Fourier coefficients, therefore the same is true for $F(\tau)$.\ And by choosing $k$ large enough we can ensure that $F(\tau)$ is holomorphic at the cusps, hence it is a~holomorphic vector-valued modular form.

\medskip
Now we may  apply Theorem~1.2 of~\cite{FM1}, which says that if $M(\rho)$ 
contains a~single nonzero vector-valued modular form with bounded denominators, then $\rho$ is modular.
The statement of the Theorem follows.
 \end{proof}
 
 \begin{rmk}\label{rem1}
 There are exactly 54 equivalence classes of two-dimensional \textit{irreducible} representations $\rho$ that satisfy the conclusions 
 of Theorem~\ref{thmmodular}. They are explicitly listed in Tables 1--5 \cite{M1}.
 \end{rmk}
 
 \subsection{The monic MLDE}
 In this Subsection we take up consideration of the MLDE~\eqref{ord2} which has $\frak{ch}_V$ as its solution space.
 It will be convenient to deal with the \textit{normalized} vector-valued modular form of weight~0
\begin{eqnarray*}
W_0(\tau):=\left(\begin{array}{c} f_1 \\ f_2\end{array}\right),
\end{eqnarray*}
whose components comprise a~fundamental system of solutions of the MLDE. Thus
\begin{eqnarray*}
f_1(\tau){:=}Z_V(\tau){=}q^a{+}\cdots, \quad f_2(\tau){:=}(\dim M_h)^{-1}Z_M(\tau){=}q^b{+}\cdots,
\end{eqnarray*}
where $a$, $b$ are rational numbers satisfying $a{:=}{-}c/24$ and $ b{:=}h{-}c/24$.

\begin{lem}\label{lemindroots} 
We have $a{+}b{=}1/6$ and $ab{=}{-}k_1$.
\end{lem}
\begin{proof} 
We know~\cite{M3} that $q{=}0$ is a~regular singular point of~\eqref{ord2}, 
and that the corresponding indicial roots are $a$ and $b$.\ The indicial equation is easily found~\cite{M3} to be 
$x^2-x/6-k_1{=}0$, and the Lemma follows.\ Actually, our main need will be the formula  $a+b{=}1/6$, which also follows immediately from the
modular Wronskian argument, cf.~\cite{M3}, Theorems~3.7 and~4.3.
\end{proof}

Let us set
\begin{eqnarray*}
\rho(T){:=}\left(\begin{array}{cc} e^{2\pi im_1} & 0 \\0 & e^{2\pi i m_2}\end{array}\right)\quad\text{with}\quad
0\leq m_j<1.
\end{eqnarray*}
We know by Theorem~\ref{thmmodular} that $\rho$ has finite image, and in particular $\rho(T)$ has finite order.
This implies that $m_1, m_2{\in}\QQ$. Moreover $a{\equiv}m_1, b{\equiv}m_2\pmod{\ZZ}$.

\begin{lem}\label{lemTexps} We have $m_1{+}m_2{=}7/6$.\
There are just~9  possibilities for  the \textup{(}unordered\textup{)} pair $\{m_1,\, m_2\}$ as follows:
\begin{table}[htp]
\begin{center}
\begin{tabular}{c|c|c|c|c|c|c|c|c|c|c|c|c|c|c|c|c}
$\{m_1,\, m_2\}$&$\{5/6, 1/3\}$&$\{3/4, 5/12\}$&$\{11/12, 1/4\}$&$\{23/24, 5/24\}$\\
\hline
$\{17/24, 11/24\}$&$\{53/60, 17/60\}$&$\{47/60, 23/60\}$&$\{41/60, 29/60\}$&$\{59/60, 11/60\}$
\end{tabular}
\end{center}
\caption{Values of $m_1$ and $m_2$}
\label{default}
\end{table}
\end{lem}
\begin{proof} 
Since $a{+}b{\equiv} m_1{+}m_2{\pmod\ZZ}$ and $0{\leq} m_1{+}m_2{<}2$,
after Lemma~\ref{lemindroots} the only possibilities are $m_1{+}m_2{=}1/6$ or $7/6$.\ On the other hand, by Remark~\ref{rem1} 
there are just 54~isomorphism classes of irreducible $\rho$ (irreducibility of $\rho$ is one of our hypotheses), 
and they are uniquely determined by the pair $\{m_1,\, m_2\}$.\ Indeed, Tables 1--5 in~\cite{M1} list all 54~possibilities, 
and we observe from these Tables that the case $m_1{+}m_2{=}1/6$ never occurs and that there are just nine choices of 
$\rho$ with $m_1{+}m_2{=}7/6$. The corresponding pairs $\{m_1,\, m_2\}$ are as listed, and the Lemma is proved.
\end{proof}

\subsection{Hypergeometric series}\label{HyperG}
In this Subsection we show that $Z_V(\tau)$ and $Z_M(\tau)$ are given by hypergeometric series evaluated 
at the level 1~hauptmodul~$K$, as in the statement of  Main Theorem~\ref{thm1}.\ We follow the arguments of~\cite{FM1}.\  
First rewrite~\eqref{ord2} as follows:\
\begin{eqnarray}\label{MLDE1}
\theta^2(f)-\tfrac{1}{6}E_2\theta(f)-k_1E_4f=0,
\end{eqnarray}
where 
\begin{eqnarray*}
\theta:=q \frac{d}{dq}.
\end{eqnarray*}
We then switch variables, from $q$ to $j$.\  As computed in~\cite{FM1}, we obtain
\begin{eqnarray}\label{MLDE2}
\frac{d^2f}{dj^2}+\frac{7j-4\cdot1728}{6j(j-1728)}\frac{df}{dj}{-} \frac{k_1}{j(j{-}1728)}f{=}0,
\end{eqnarray}
which is nothing but the Gauss normal form (cf.~\cite{M})
\begin{eqnarray}\label{GNF}
\frac{d^2f}{dJ^2}{+}\frac{C{-}(A{+}B{+}1)J}{J(1{-}J)}\frac{df}{dJ}{-}\frac{AB}{J(1{-}J)}f{=}0,
\end{eqnarray}
with
\begin{eqnarray*}
J{=}K^{-1},\;\,  C{=}\frac{2}{3},\; \, A{+}B{=}\frac{1}{6},\;\,  AB{=}{-}k_1{=}ab.
\end{eqnarray*}

Note that $A$, $B$ satisfy the same equations as $a$, $b$ (Lemma~\ref{lemindroots}), so that we may, and shall, take
$A{=}a, B{=}b$. The pair of fundamental solutions of~\eqref{GNF}) at $\infty$ are then the hypergeometric series
\begin{equation}\label{hypergeo2}
f_1{:=} K^a \cdot {_2}F_1(a, a{+}1/3, 2a{+}5/6; K),\quad 
f_2{:=} K^b\cdot {_2}F_1(b, b{+}1/3, 2b{+}5/6; K).
\end{equation}

\subsection{Bounds for $c$ and $m$ }
In this Subsection we show that there are \textit{only a finite number of possibilities} for the central charge $c$ of $V$ 
and the integer $m$ defined to be the dimension of the first nontrivial graded piece $V_1$ of $V$.\ 
To achieve this we will use the description~\eqref{hypergeo2} of $Z_V(\tau)$ and $Z_M(\tau)$ as a~hypergeometric series.\
 
 \medskip
We continue with previous notation, so that $a{=}{-}c/24$, $b{=}h{-}c/24$ and 
\begin{equation*}
f_1{=}Z_V(\tau){=}q^a{+}\cdots,\quad f_2{=}Z_M(\tau){=}q^b{+}\cdots
\end{equation*} 
(up to an overall integral scalar).\ Using the hypergeometric description~\eqref{hypergeo2} and the explicit formula~\eqref{hypergeo1} we find that, up to an overall scalar,
\begin{equation*}
\begin{split}
&f_1(\tau){\sim}\left(12^3q(1-744q+356652q^2+\cdots)\right)^a\\
&\times\left\{1+\frac{12^3a(a+1/3)}{2a+5/6}q+\left(-\frac{12^3\cdot744a(a+1/3)}{2a+5/6}
+\frac{12^6a(a+1)(a+1/3)(a+4/3)}{2(2a+5/6)(2a+11/6)}\right)q^2+\cdots\right\}.
\end{split}
\end{equation*}
 
\begin{rmk} 
This series \textit{does} converge.\  Indeed, it will converge as long as $2a+5/6$  is not a~nonpositive integer, and this follows from Lemma~\ref{lemTexps}. 
 \end{rmk}

To write the first factor as a~$q$-expansion, we use Newton's binomial expansion
\begin{eqnarray*}
(1+X)^a= \sum_{k=0}^{\infty} {a\choose k} X^k
\end{eqnarray*}
with $X:=-744q+356652q^2+\cdots$ to obtain
\begin{equation*}
\begin{split}
 f_1&= q^a\left\{1-744aq+\left(356652a+\frac{744^2a(a-1)}{2}\right)q^2+\cdots\right\}\\
 &\;\times\left\{1+\frac{12^3a(a+1/3)}{2a+5/6}q+\left(-\frac{12^3\cdot744a(a+1/3)}{2a+5/6}+\frac{12^6a(a+1)(a+1/3)(a+4/3)}{2(2a+5/6)(2a+11/6)}\right)q^2+\cdots\right\}\\
 &=q^a\left\{1+24a\left(\frac{(6a+2)}{12a+5}72-31\right)q+\cdots\right\}
  \end{split}
 \end{equation*}
The first nontrivial coefficient of $f_1$ is therefore
 \begin{eqnarray*}
 \dim V_1{=:}m{=}{-}c\left(\frac{(8{-}c)}{(10{-}c)}36{-}31\right){=}\frac{c(5c{+}22)}{10{-}c}.
 \end{eqnarray*}
 
This formula is known.  It appears, for example, on P.~368 of~\cite{TV}, where it also arises from consideration 
of the MLDE~\eqref{ord2}, but instead of hypergeometric series Tuite and Van use special properties of the \textit{exceptional} VOAs 
that they are studying.
 
 \medskip
The previous display is equivalent to $5c^2+(22+m)c-10m=0$  (note that $c{=}{-}24a{\neq}10$),
so $c{=} \left({-}(22{+}m)\pm\sqrt{m^2{+}484{+}244m}\right)/10$.

\medskip
Because $V$ is strongly regular then $c{\in}\QQ$ (see  \cite{DLM}), so there is an integer $s$ such that
$s^2{=}m^2{+}244m{+}484{=}(m{+}122)^2{-}120^2$.
Thus
\begin{eqnarray}\label{triple}
&&s^2+120^2=(m+122)^2
\end{eqnarray}
and the solutions correspond to  \textit{Pythagorean triples} $(s, 120, m+122)$. (A~triple of integers that may serve as lengths of sides of a~(possibly degenerate) right triangle.)

\medskip
There is an~old and well-known algorithm (Euclid) that gives a parameterization of all Pythagorean triples.\  In our case there are only finitely many nonnegative integral pairs $(s, m)$ that solve~\eqref{triple}, and we may use Euclid's algorithm to readily find them all.\ They are set out in Table~\ref{Table:smca}.\ We content ourselves by listing the resulting pairs.\
We also list the corresponding pairs of possible values of $c=-(m+22)\pm s)/10$ and $a=-c/24$, which we will need.
 
\begin{table}[hbtp]
\begin{center}
\begin{tabular}{c|c|c|clc}
\hline
$s$&$m$&$c$&$a$\\
\hline
$3599$&$3479$&$-710$,\, $49/5$&$355/12$,$-49/120$\\
$896$&$782$&$ -170$,\, $46/5$&$  85/12$,\, $-23/60$\\
$391$&$287 $&$-70 ,\,41/5$&$35/12$,\, $-41/120$\\
$209$&$119$&$ -35$,\, $34/5$&$ 35/24$,\, $-17/60$\\
$119$&$47 $&$ -94/5$,\, $5$&$47/60$,\, $-5/24$\\
$64$&$14$&$-10$,\, $14/5$&$5/12$,\, $-7/60$\\
$1798$&$1680 $&$ -350$,\, $48/5$&$ 175/12$,\, $-2/5$\\
$442$&$336$&$-80$,\, $ 42/5$&$10/3$,\, $-7/20$\\
$182$&$96$&$-30$,\, $32/5 $&$ 5/4$,\, $-4/15$\\
$22$&$ 0$&$0$,\, $-22/5$&$0$,\, $11/60$\\
$1197$&$1081$&$ -230$,\, $47/5$&$115/12,\,-47/120$\\
$288$&$190$&$  -50$,\, $38/5$&$ 25/12$,\, $-19/60$\\
$27 $&$1$&$ -5$,\, $2/5$&$5/24$,\, $-1/60$\\
$715$&$ 603$&$  -134 ,\,9$&$ 67/12$,\, $-3/8$\\
$35$&$ 3 $&$ -6$,\, $1$&$ 1/4$,\, $-1/24$\\
$594$&$ 484$&$   -110$,\, $44/5$&$ 55/12$,\, $-11/30$\\
$126$&$52$&$ -20,\,26/5$&$ 5/6$,\, $-13/60$\\
$350$&$248$&$-62$,\, $8$&$  31/12$,\, $-1/3$\\
$50 $&$8$&$ -8$,\, $2$&$1/3$,\, $-1/12$\\
$225$&$33$&$ -38$,\, $7$&$19/12$,\, $-7/24$\\
$160$&$78$&$-26$,\, $6$&$13/12$,\, $-1/4$\\
$0$&$ 28$&$ -14$,\, $4$&$7/12$,\, $-1/6$
\end{tabular}
\end{center}
\caption{Values of $s$, $m$, $c$ and $a$}
\label{Table:smca}
\end{table}

We now compare the values of $a$ in the fourth column of Table~\ref{Table:smca} with the values of $m_j$ in Lemma~\ref{lemTexps}. For we know that there is an~index $j$ such that $a{\equiv}m_j\pmod\ZZ$.  A~number of values of~$a$ do not survive this test, and those that do are listed in Table~\ref{Table:smca2}.

\begin{table}[hbtp]
\begin{center}
\begin{tabular}{c|c|c|clc}
\hline
$s$&$m$&$c$&$a$\\
\hline
$391$&$ 287  $&$-70$&$35/12$\\
 $ 209$&$  119  $&$-35 $&$  35/24$\\
 $119$&$ 47 $&$ -94/5 $&$47/60$\\
$ 64$&$ 14 $&$ -10, 14/5$&$  5/12,-7/60$\\
$ 442$&$  336 $&$-80 $&$ 10/3$\\
$ 182$&$   96 $&$ -30$&$  5/4$\\
$ 22$&$0$&$-22/5$&$ 11/60$&$ $\\
$ 288$&$  190  $&$38/5 $&$ -19/60$\\
 $27$&$1$&$-5, 2/5$&$5/24,-1/60$\\
$ 35  $&$ 3 $&$ -6, 1$&$ 1/4,-1/24$\\
$ 126$&$ 52 $&$ 26/5$&$-13/60  $\\
$ 50  $&$  8 $&$ -8, 2 $&$ 1/3,-1/12  $\\
$ 225  $&$ 133 $&$ 7$&$-7/24 $\\
$ 160$&$  78 $&$  6 $&$ -1/4$\\
$  90$&$ \ 28 $&$   4$&$  -1/6$\\
$126$&$52$&$-20$& $5/6$
\end{tabular}
\end{center}
\caption{Values of $s$, $m$, $c$ and $a$}
\label{Table:smca2}
\end{table}

Next we record, for each $a$-value in Table~\ref{Table:Expansion} an~initial segment of the $q$-expansion 
of $f_1{=}K^a\cdot{}_2F_1(a, a{+}1/3, 2a{+}5/6; K)$.\ These can be found in Table~\ref{Table:Expansion}.
\renewcommand{\arraystretch}{1.5}
\begin{table}[htp]
\caption{$q$-expansion of $f_1$}
\begin{center}
\begin{tabular}{c|l}
$a$&\qquad\qquad\qquad\qquad\qquad\qquad\qquad\qquad$f_1$\\
\hline
$35/12$&$q^{35/12}(1+287q+\frac{847903}{23}q^2+\cdots)$\\
$35/24$&$q^{35/24}(1+119q+\frac{113358}{19}q^2+\cdots)$\\
$47/60$&$q^{47/60}(1+47q+\frac{15369}{17}q^2+\cdots)$\\
$5/12$&$q^{5/12}(1+14q+92q^2+456q^3+1848q^4+6580q^5+21141q^6+62806q^6+174777q^7+\cdots)$\\
$-7/60$&$q^{-7/60}(1+14q+42q^2+140q^3+350q^4+840q^5+1827q^6+3858q^7+7637q^8+\cdots)$\\
$10/3$&$q^{10/3}\left(1 + 336 q + \frac{868136}{17} q^2) + \frac{1541266112}{323} q^3) 
+ \frac{5323642484}{17} q^4 + \frac{264979509920}{17} q^5+\cdots\right)$\\
$5/4$&$q^{5/4}\left(1+96q+\frac{49869}{13}q^2+\cdots\right)$\\
$5/6$&$q^{-6/5}(1+1292 q+701246 q^2+207599288 q^3+36592296829 q^4+3988939885028 q^5+\cdots)$\\
$11/60$&$q^{11/60}(1+q^2+q^3+q^4+q^5+2q^6+2q^7+3q^8+\cdots)$\\
$-19/60$&$q^{-19/60}(1+190q+2831q^2+22306q^3+129276q^4+611724q^5+2511667q^6+\cdots)$\\
$5/24$&$q^{5/24}(1+q+3q^2+4q^3+7q^4+10q^5+17q^6+23q^7+35q^8+\cdots)$\\
$-1/60$&$q^{-1/60}(1+q+q^2+q^3+2q^4+2q^5+3q^6+3q^7+4q^8+\cdots)$\\
$1/4$&$q^{1/4}(1+3q+9q^2+19q^3+42q^4+81q^5+155q^6+276q^6+486q^7+\cdots)$\\
$-1/24$&$q^{-1/24}\left(1+3 q+4 q^2+7 q^3+13 q^4+19 q^5+29 q^6+43 q^7+62 q^8+\cdots\right)$\\
$-13/60$&$q^{-13/60}(1+52q+377q^3q^2+1976q^3+7852q^4+27404q^5+84981q^6+243230q^7+\cdots$)\\
$1/3$&$q^{1/3}(1+8q+36q^2+128q^3+394q^4+1088q^5+2776q^6+6656q^7+15155q^8+\cdots\cdots)$\\
$-1/12$&$q^{-1/12}(1+8q+17q^2+46q^3+98q^4+198q^5+371q^6+692q^7+1205q^8+\cdots)$\\
$-7/24$&$q^{-7/24}(1+133q+1673q^2+11914q^3+63252q^4+278313q^5+1070006q^6+\cdots$\\
$-1/4$&$q^{-1/4}(1+78q+729q^2+4382q^3+19917q^4+77274q^5+264664q^6+827388q^7+\cdots)$\\
$ -1/6$&$q^{-1/6}(1+28q+134q^2+568q^3+1809q^4+5316q^5+13990q^6+34696q^7+\cdots)$
\end{tabular}
\end{center}
\label{Table:Expansion}
\end{table}%
\renewcommand{\arraystretch}{1.0}

\medskip
Thus the cases $a{=}35/12,\, 35/24,\, 47/60, \, 5/4, \,10/3$ are eliminated because then $f_1$ has coefficients 
that are \textit{not} integers.  On the other hand, in the case $a{=}5/6$ we find that (up to an overall scalar) we have
\begin{eqnarray*}
f_2{\sim} q^{-2/3}(1{-}272q{-}34696q^2{-}1058368q^3{-}\hdots)
\end{eqnarray*}
so that this possibility is eliminated on account of the negative coefficients.\ What remains is the list of possibilities 
in Table~\ref{Table:residual}, where we also include the corresponding values of $b$ and the \textit{effective central charge}~$\tilde{c}$.\ This invariant is discussed in Subsection~\ref{SStildec}, and calculated using Lemma~\ref{lemtildec}.\ 
(Consideration of $f_2$ as in the case $a{=}5/6$ does not yield any useful information in these cases.)

\begin{table}[htbp]
\caption{Residual possibilities}
\begin{center}
\begin{tabular}{c|c|c|c|c|c}
$s$&$m$&$c$&$a$&$b$&$\tilde{c}$\\
\hline
$64$&$14$&$ -10, 14/5$&$5/12$,\,$-7/60$&$ -1/4, 17/60$&$6$, \,$14/5$\\
$22$&$0$&$-22/5$&$11/60$&$-1/60$&$2/5$\\
$288$&$190$&$38/5$&$-19/60$&$29/60$&$38/5$\\
$27$&$1$&$-5$,\,$2/5$&$5/24$,\, $-1/60$&$-1/24$, \,$11/60$&$1$, \,$2/5$\\
$35$&$3$&$-6$,\, $1$&$1/4$, \,$-1/24$&$-1/12$, \,$5/24$&$2$,\, $1$ \\
$126$&$52$&$26/5$&$-13/60$&$23/60$&$26/5$\\
$50$&$8$&$-8$, \,$2$&$1/3$, $-1/12$&$ -1/6$,\,$ 1/4$&$4$, \,$2$ \\
$225$&$133$&$7$&$-7/24$&$11/24$&$7$\\
$160$&$78$&$6$&$-1/4\ $&$ 5/12$&$6$\\
$90$&$28$&$4$&$ -1/6$&$1/3$&$4$
\end{tabular}
\end{center}
\label{Table:residual}
\end{table}

\subsection{The effective central charge $\tilde{c}$}\label{SStildec}
In this Subsection we will show that some additional cases listed in Table~\ref{Table:residual} \textit{cannot} correspond to 
strongly regular VOAs.\ To do this we make use of the \textit{effective central charge} $\tilde{c}$ of $V$ defined as follows:\
\begin{eqnarray*}
\tilde{c}{:=}c{-}24h_{\min}
\end{eqnarray*}
where $h_{\min}$ is defined to be the \textit{minimum} of~0 and $h$. 
(Recall that $h$ is the conformal weight of the irreducible $V$-module $M$.) 
By (1.3) in~\cite{DM1}, a strongly regular VOA necessarily satisfies $\tilde{c}{>}0$. In the present situation we have
\begin{lem}\label{lemtildec} 
Exactly one of ${-}24a$, ${-}24b$ is \textit{positive}, and this is equal to~$\tilde{c}$.
\end{lem}
\begin{proof}
First observe from Table~\ref{Table:smca} that exactly one of $a$, $b$ is negative.\  If $h{\geq}0$ then $h_{\min}{=}0$ 
and then $\tilde{c}{=}c{=}{-}24a$.\  On the other hand, if $h{<}0$ then $h_{\min}{=}h$ and furthermore \\
$\tilde{c}{=}c{-}24h{=}c{-}24(b{+}c/24){=}{-}24b$.\  The Lemma follows.
\end{proof}

In the following omnibus Theorem we collect some further results, gleaned from~\cite{DM1}, \cite{DM2} and~ \cite{M2},
having to do with the effective central charge $\tilde{c}$ in an arbitrary strongly regular VOA.
\begin{thm}\label{thmomni} 
Assume that $V=\CC\mathbf{1}{\oplus}V_1{\oplus}\cdots$ is a~strongly regular VOA of central charge~$c$.
Let $\ell$ be the \textit{Lie rank} of the Lie algebra $V_1$ \textup{(}dimension of a maximal abelian subalgebra of $V_1$\textup{)}.
Then the following hold\textup{\,:}\\
\textup{(a)}~The \textit{Lie algebra} $V_1$ is \textit{reductive}.\\
\textup{(b)}~$\ell\leq\tilde{c}$\,.\\
\textup{(c)}~If $c{=}\ell{=}\tilde{c}$ then $V$ is \textit{isomorphic} to a lattice theory $V_{\Lambda}$ for some even, 
positive-definite lattice $\Lambda$ of rank $\ell$.\\
\textup{(d)}~Suppose that $\tilde{c}{<}\ell{+}1$.\ Then $c{=}\ell{+}c_{p, q}$ where $c_{p, q}$ is a central charge 
in the \textit{Virasoro discrete series}.\\
\textup{(e)}~If $L{\subseteq} V_1$ is a~\textit{Levi factor} of $V_1$ then the subVOA $U{:=}\langle L\rangle{\subseteq}V$ generated by $L$ is isomorphic to a~tensor product of affine algebras $L(\frak{g}_i, k_i)$ at \textit{positive integral} levels $k_i$. 
\textup{(}$U$ and $V$ may have \textit{different} conformal vectors.\textup{)}\\
\textup{(f)}~If $\tilde{c}{=}2/5$ then $V$ is isomorphic to the Virasoro (Yang-Lee) model with~$c{=}{-}22/5$.
\end{thm}
\begin{proof} Parts (a), (b) and (c) correspond to  Theorems~1.1, 1.2 and 1.3 respectively of~\cite{DM1}.\ Part (d) is an~immediate consequence of Theorem~7 of~\cite{M2}.\ Part (e) follows from Theorem~1.1 of~\cite{DM2} and Theorem~3 of~\cite{M2}, 
while (f) is a~restatement of Corollary~9 of~\cite{M2}.
\end{proof}
 
\begin{table}[htbp]
\begin{center}
\begin{tabular}{c|cccccccccccc}
$\ell$&$1$&$2$&$3$&$4$&$5$&$6$&$7$&$8$&$9$&$10$\\
\hline
$A_\ell$&$3$&$8$&$15$&$24$&$35$&$48$&$63$&$80$&$99$&$120$\\
$B_\ell$&$3$&$10$&$21$&$36$&$55$&$78$&$105$&$136$&$171$&$210$\\
$C_\ell$&$3$&$10$&$21$&$36$&$55$&$78$&$105$&$136$&$171$&$210$\\
$D_\ell$&$3$&$6$&$15$&$28$&$45$&$66$&$ 91$&$120$&$15$3&$190$\\
\hline
$F_4$&$52$&&&&&&&&\\
$G_2$&$14$&&&&&&&&\\
$E_6$&$78$&&&&&&&&&\\
$E_7$&$133$&&&&&&&&&\\
$E_8$&$248$&&&&&&&&&\\
\end{tabular}
\end{center}
\caption{Dimensions of simple Lie algebras}
\label{Table:DIML}
\end{table}

\subsection{Proof of Main Theorem 1}\label{SSPMT}
In this Subsection we complete the proof of  Main Theorem~\ref{thm1}.\ This involves a~more detailed consideration 
of the possibilities listed in Table~\ref{Table:residual} based on the results of Theorem~\ref{thmomni}.\ The list of low-dimensional simple Lie algebras in Table \ref{Table:DIML} is also useful.\
First we deal with the 8~known cases.

\medskip\noindent
\textbf{Case $c{=}1$.}\
From Table~\ref{Table:residual} we have $m{=}3, c{=}\tilde{c}{=}1$.
By Theorem~\ref{thmomni}(a) and (b) $V_1$ is a~reductive Lie algebra of dimension $m{=}3$ and Lie rank $\ell{\leq}1$.
Thus we must have $V_1\cong\frak{sl}_2$, so that $\ell{=}1$. Now Theorem~\ref{thmomni}(c) applies and establishes 
that $V{\cong} V_{A_1}{\cong} L(A_1, 1)$.

\medskip\noindent
\textbf{Case $c{=}2$.}\
This is similar to the previous Case. We have $m{=}8$ and $c{=}\tilde{c}{=}2$, so $V_1$ is a~reductive Lie algebra of dimension~8 
and  $\ell{\leq}2$. The only possibility is $V_1{\cong}\frak{sl}_3$, 
and we can conclude with Theorem~\ref{thmomni}(c) once more that $V{\cong}V_{A_2}{\cong} L(A_2, 1)$.

\medskip\noindent
\textbf{Case $c{=}4$.}\
Here, $V_1$ is a~reductive Lie algebra of dimension~28 and Lie rank $\ell{\leq}\tilde{c}{=}c{=}4$.
By the Cartan-Killing classification of semisimple Lie algebras one checks that the only possibility is either 
$V_1{\cong} \frak{so}_8$ or $V_1{\cong}G_2\oplus G_2$, and by Theorem~\ref{thmomni}(c) 
we obtain $V{\cong} L_{D_4} {\cong}  L(D_4, 1)$.

\medskip\noindent
\textbf{Case $c{=}6$.}\
Here, $V_1$ is a reductive Lie algebra of dimension~78 and  Lie rank $\ell{\leq}\tilde{c}{=}c{=}6$.
By the Cartan-Killing classification of semisimple Lie algebras the possibilities are $V_1{\cong}\frak{e}_6, \frak{sp}_{12}$ 
or $\frak{so}_{13}$.\ In each  case we have $\tilde{c}{=}c{=}\ell$ and by Theorem~\ref{thmomni}(c) 
we obtain $V\cong L_{E_6}\cong L(E_6, 1)$.

\medskip\noindent
\textbf{Case $c{=}7$.}
Here, $V_1$ is a~reductive Lie algebra of dimension 133 and Lie rank $\ell{\leq}\tilde{c}=c=7$.
By the Cartan-Killing classification of semisimple Lie algebras the only possibility is $V_1{\cong} \frak{e}_7$ 
and by Theorem~\ref{thmomni}~(c) we obtain $V{\cong} L_{E_7}{\cong} L(E_7, 1)$.

\medskip
This deals with the cases of affine algebras with simply-laced root systems.\ In other cases the argument is a bit more complicated:

\medskip\noindent
\textbf{Case $c{=}14/5$.}
Here, $V_1$ is a~reductive Lie algebra of dimension~14 and Lie rank $\ell{\leq}\tilde{c}=c=14/5$.
Thus $\ell{\leq}2$ and by the Cartan-Killing classification of semisimple Lie algebras the only possibility is $V_1{\cong}\frak{g}_2$. 
Now from Tables~\ref{Table:Expansion} and~\ref{Table:residual}, the character $Z_V(\tau)$ is uniquely determined from the hypotheses of the main Theorem together with the numerical restrictions $c{=}14/5$ and $\dim V_1{=}14$.
Because the affine algebra $L(G_2, 1)$ also satisfies these conditions then it follows $f_1{=}Z_V(\tau){=}q^{-7/60}+\cdots$ 
as given in Table~\ref{Table:Expansion} is exactly the graded character of $L(G_2, 1)$.

\medskip
On the other hand, if $U{:=}\langle V_1\rangle$ is as in the statement of Theorem~\ref{thmomni}(e) then that result shows that
$U{\cong} L(G_2, k)$ for some positive integer~$k$.\  It follows from the last paragraph that the graded character of $L(G_2, k)$ 
is \textit{majorized} by that of $L(G_2, 1)$ in the following sense:\ \textit{every} coefficient in the graded character of $L(G_2, k)$ is \textit{no greater} than the corresponding coefficient in the graded character of $L(G_2, 1)$.

\medskip
Now $L(G_2, k)$ is constructed as a~graded quotient of the universal VOA $M(G_2,k)$ associated with the Lie algebra $G_2$, 
and the (unique) maximal submodule of $M(G_2,k)$ is generated by $e_{\theta}(-1)^{k+1}\mathbf{1}$, where $e_{\theta}$ 
is the longest root (cf.~\cite[Chapter 6.6]{LL}).\ Because the graded dimension of $L(G_2, k)$ is \textit{majorized} by that of $L(G_2, 1)$ in the sense of the previous paragraph, it follows that~$k{=}1$.

\medskip\noindent
\textbf{Case $c{=}26/5$.}\
Here, $V_1$ is a~reductive Lie algebra of dimension 52 and Lie rank $\ell{\leq}\tilde{c}{=}c{=}26/5$.\ Thus $\ell{\leq}5$ 
and by the Cartan-Killing classification of semisimple Lie algebras the only possibility is $V_1{\cong}\frak{f}_4$.\
The rest of the argument proving that $V{\cong} L(F_4, 1)$ is completely parallel to that of the previous case, except that of course we replace $G_2$ with $F_4$.

\medskip
The remaining entry in Table~\ref{Table:residual} corresponding to a known VOA is the following:

\medskip\noindent
\textbf{Case $c{=}{-}22/5$.}\
In this Case we have $\tilde{c}{=}2/5$ from Table~\ref{Table:smca2}, therefore by Theorem~\ref{thmomni}~(f), $V$ is the Virasoro
VOA in the discrete series with $c{=}{-}22/5$.\  Alternatively, we have $\dim V_1{=}0$ from Table~\ref{Table:smca2}, 
whence the identification of $V$ follows from the characterization of the same Virasoro algebra given in~\cite{ANS}.

\medskip
Next we show by arguments similar to those already used that the cases with $c{=}2/5,\, 38/5,\, {-}5$ \textit{do not} correspond to strongly regular VOAS.

\medskip\noindent
\textbf{Case $c{=}2/5$.}\
Here, Table~\ref{Table:residual} informs us that $\tilde{c}{=}2/5$ and $\dim V_1{=}1$.  It follows that $V{=}\CC$ and therefore 
$\ell{=}1{>}\tilde{c}$, contradicting Theorem~\ref{thmomni}(b).  Alternatively, we may apply Theorem~\ref{thmomni}~(e) to see 
that $V$ is the Yang-Lee model with $c{=}{-}22/5$, a~contradiction.

\medskip\noindent
\textbf{Case $c{=}38/5$.}\
From Table~\ref{Table:residual} and various parts of Theorem~\ref{thmomni}, we find that $V_1$ is a~reductive Lie algebra 
of dimension~190 and Lie rank $\ell{\leq}7$.  But there is \textit{no} such Lie algebra, as we can see using Table~\ref{Table:DIML}.
So this Case cannot occur.

\medskip\noindent
\textbf{Case $c{=}{-}5$.}\
From Table~\ref{Table:residual}, $\tilde{c}{=}1$ and $\dim V_1{=}1$, so $\ell{=}1$.  So Theorem~\ref{thmomni}~(d) applies, 
and tells us that $c{=}1{+}c_{p, q}$. This is impossible because $c{-}1{=}{-}6$ is \textit{never} equal to any $c_{p, q}$.

\section{Proof of Main Theorem 2}\label{SPMT2}
In this Section we give the proof of Main Theorem 2.\ Essentially, we must handle the three remaining  cases, where $c=-6,\, -8$ 
and $-10$.\ We will show that they cannot occur under the assumption that $V$ and $M$ are the only simple $V$-modules.\
The methods employed in Section~\ref{SSPMT} are less effective when dealing with these  cases.\
Instead we will use the modularity Theorem~\ref{thmmodular} coupled with the fact~\cite{Hu} that the $S$-matrix is \textit{symmetric}.

\medskip\noindent
\textbf{Case $c=-6$.}\
We will need the explicit identification of $f_1$ and $f_2$ as modular functions of level~12. 
(The level is the least common  of the denominators of $a=1/4$ and $b=-1/12$.) In fact, we have
\begin{eqnarray*}
f_1(\tau){=}\Delta_3(\tau)/\eta(\tau)^2,\ \ f_2(\tau){=}I_3(\tau)/\eta(\tau)^2,
\end{eqnarray*}
where
\begin{equation*}
\Delta_3(\tau){:=}\eta(3\tau)^3/\eta(\tau),\quad I_3(\tau){:=}1+6\sum_{n=1}^{\infty}\sum_{d|n} \left(\frac{d}{3}\right)q^n,
\end{equation*}
and $\left(\tfrac{d}{3}\right)$ is the Legendre symbol.\ (We use the provisional notation $\Delta_3$ and $I_3$ as
there is no standard way to denote the corresponding modular forms.)\ This can be checked in various ways:
(a) show that the indicated modular forms solve the MLDE (\ref{ord2});
(b) check that $\Delta_3(\tau)$ and  $I_3(\tau)$ are holomorphic modular forms of weight~1 and level~12 and that the first few terms 
of their q-expansions agree with those of $\eta(\tau)^2f_1(\tau)$ and $\eta(\tau)^2f_2(\tau)$ respectively.

\medskip
Using standard transformation laws, we find that
\begin{equation*}
\begin{pmatrix} f_1\\ f_2\end{pmatrix}\Big|_0 S=
-\frac{1}{\sqrt{3}}\begin{pmatrix}1&-\frac{1}{3}\\
6&1\end{pmatrix}
\begin{pmatrix} f_1\\ f_2\end{pmatrix}.
\end{equation*}

Thus the $S$-matrix for $V$ is visibly \textit{not} symmetric, and therefore $V$ cannot exist.

\medskip\noindent
\textbf{Case 2 $c{=}{-}8$.}\
We proceed as in Case~1.\  We find that
\begin{equation*}
f_1= \frac{\eta_(2\tau)^8}{\eta(\tau)^8},\quad f_2=\frac{2E_2(2\tau)^2-E_2(\tau)}{\eta(\tau)^4}
\end{equation*}
and 
\begin{equation*}
\left.\begin{pmatrix} f_1\\ f_2\end{pmatrix}\right|_0S=\frac{1}{2}
\begin{pmatrix}
-1&\frac{1}{8}\\
24&1
\end{pmatrix}
\begin{pmatrix} f_1\\ f_2\end{pmatrix}\,.
\end{equation*}
We see that the $S$-matrix is not symmetric.

\medskip\noindent
\textbf{Case  $c{=}{-}10$.}\
Proceed as in Cases~1 and~2.  We find that
\begin{equation*}
f_1=\frac{I_3(\tau)^3\Delta_3(\tau)^3}{\eta(q)^6}\,,\quad f_2=\frac{I_3(\tau)^3+54\Delta_3(\tau)^3}{\eta(q)^6}
\end{equation*}
and 
\begin{equation*}
\left.\begin{pmatrix} f_1\\ f_2\end{pmatrix}\right|_0S=\tfrac{1}{\sqrt{3}}
\begin{pmatrix}
-1&-\tfrac{1}{27}\\
54&1
\end{pmatrix}
\begin{pmatrix} f_1\\ f_2\end{pmatrix}\,.
\end{equation*}
Once again, the $S$-matrix is not symmetric.

\medskip
This completes the proof that the three cases where $c{=}{-}6,\,{-}8,\,{-}10$ \textit{cannot} occur.\
Now our Main Theorem 2 follows from Main Theorem 1.

\end{document}